\documentclass{article}
\usepackage{amsmath}
\usepackage{amssymb}
\usepackage{amscd}
\usepackage{multicol}
\usepackage{color}
\usepackage{amsthm,mathtools}
\mathtoolsset{showonlyrefs}

\begin{document}
\def\ni{\noindent}
\def\t{\theta}
\def\O{\Omega}
\def\S{\Sigma}
\def\e{\epsilon}
\def\lra{\longrightarrow}
\def\R{{\mathbb R}}
\def\N{{\mathbb N}}
\def\Z{{\mathbb Z}}
\def\RR{{{\mathbb R}}^2}
\def\MS{M^{2{\rm x}2}_s}
\def\N{{\bf N}}
\def\l{\lambda}
\def\LL{${\cal L}$}
\def\E{{\cal E}}
\def\a{{\alpha}}
\def\A{A_{\a}}
\def\ta{\t^{\a}}
\def\rot{{\rm rot}}

\newcommand{\vs}[1]{\vskip #1pt}

\newtheorem{Theorem}{Theorem}[section]
\newtheorem{Definition}[Theorem]{Definition}
\newtheorem{corollary}[Theorem]{Corollary}
\newtheorem{proposition}[Theorem]{Proposition}
\newtheorem{examples}[Theorem]{Esempi}
\newtheorem{example}[Theorem]{Example}
\newtheorem{lemma}[Theorem]{Lemma}
\newtheorem{remark}[Theorem]{Remark}

\catcode`\@=11


   \renewcommand{\theequation}{\thesection.\arabic{equation}}
   \renewcommand{\section}%
   {\setcounter{equation}{0}\@startsection {section}{1}{\z@}{-3.5ex  
plus -1ex
    minus -.2ex}{2.3ex plus .2ex}{\Large\bf}}

\title{\bf Unbounded solutions to systems of differential equations at resonance\thanks{Under the auspices of GNAMPA-I.N.d.A.M., Italy.}
}
\author{A. Boscaggin, W. Dambrosio and D. Papini}

\date{}
\maketitle

\begin{center}
Dedicated to the memory of Russell Johnson.
\end{center}

\vs{12}

\begin{quote}
\small
{\bf Abstract.} We deal with a weakly coupled system of ODEs of the type
$$
x_j'' + n_j^2 \,x_j + h_j(x_1,\ldots,x_d) = p_j(t), \qquad j=1,\ldots,d,
$$
with $h_j$ locally Lipschitz continuous and bounded, $p_j$ continuous and $2\pi$-periodic, $n_j \in \mathbb{N}$ (so that the system is at resonance). By means of a Lyapunov function approach for discrete dynamical systems, we prove the existence of unbounded solutions, when
either global or asymptotic conditions on the coupling terms $h_1,\ldots,h_d$ are assumed.
\medbreak
\noindent
{\bf Keywords.} Systems of ODEs, unbounded solutions, resonance, Lyapunov function.

\noindent
{\bf AMS Subject Classification.} 34C11, 34C15.
\end{quote}

\section{Introduction}

In this paper, we deal with the existence of unbounded solutions for weakly coupled systems of ODEs of the type
\begin{equation}\label{sys-intro}
\begin{cases}
\, x''_1 + n_1^2 \,x_1 + h_1(x_1,\ldots,x_{d}) = p_1(t), \vspace{7pt}\\
\, x''_2 + n_2^2 \,x_2 + h_2(x_1,\ldots,x_{d}) = p_2(t), \vspace{7pt}\\
\qquad \vdots \\
\, x''_d + n_d^2 \,x_d + h_d(x_1,\ldots,x_{d}) = p_d(t),
\end{cases}
\end{equation}
where the functions $h_1,\ldots,h_d: \mathbb{R}^d \to \mathbb{R}$ are locally Lipschitz continuous and bounded and the functions
$p_1,\ldots,p_d: \mathbb{R} \to \mathbb{R}$ are continuous and periodic with the same period, say $2\pi$ for simplicity.
We will also assume that 
\begin{equation}\label{res-intro}
n_j \in \mathbb{N} \quad \mbox{ for every } j \in \{1,\ldots,d\},
\end{equation}
implying, as well-known, that the scalar equation $x''_j + n_j^2 x_j = 0$ has a nontrivial $2\pi$-periodic solution
(in fact, all its nontrivial solutions are $2\pi$-periodic). Following a popular terminology (cf. \cite{Maw07}), system \eqref{sys-intro} is thus said to be at resonance.

In the scalar case (that is, $d=1$), the problem of the existence of unbounded solutions has been considered since the nineties.
Indeed, the first result can be essentially traced back to Seifert \cite{Sei90}, establishing the existence of unbounded solutions to the equation
\begin{equation}\label{eq-intro}
x'' + n^2 x + h(x) = p(t), \qquad x \in \mathbb{R},
\end{equation}
as a consequence of a non-existence result for $2\pi$-periodic solutions by Lazer and Leach \cite{LazLea69} together with the classical Massera's theorem. Later on, sharper conclusions were obtained by Alonso and Ortega in \cite{AloOrt96}. In particular, according to \cite[Proposition 3.1]{AloOrt96}, any solution of \eqref{eq-intro} is unbounded both in the past and in the future whenever 
\begin{equation}\label{glob-intro}
2 \left( \sup h - \inf h \right) \leq \left\vert \int_0^{2\pi} p(t) e^{\textnormal{i} n t}\,dt\right\vert;
\end{equation}
moreover, due to \cite[Proposition 3.4]{AloOrt96}, any sufficiently large solution is unbounded either in the past or in the future when the global condition \eqref{glob-intro} is replaced by the (weaker) asymptotic assumption 
\begin{equation}\label{asy-intro}
2 \left\vert \max\left\{ \limsup_{x \to +\infty}h(x), \limsup_{x \to -\infty}h(x)\right\} - \min\left\{ \liminf_{x \to +\infty}h(x), \liminf_{x \to -\infty}h(x)\right\} \right\vert < \left\vert\int_0^{2\pi} p(t) e^{\textnormal{i} n t}\,dt\right\vert.
\end{equation}
Both results were proved via an abstract method, based on the use of Lyapunov-like functions, developed in the same paper. 
Generalizations of this kind of results to more general situations (like asymmetric oscillators and planar Hamiltonian systems)
were then obtained by many authors (see, among others, \cite{AloOrt98,CapDamMaWan13,CapDamWan08,Dam02,FabFon05,FabMaw00,KunKupLiu01,LiuTorQia15,MaWan13,Maw07,Yan04} and the references therein). 

Yet, as far as we know, the boundedness problem for system \eqref{sys-intro} is rather unexplored.
The aim of the present paper is to provide some results in this direction, by investigating to what extent the method developed in \cite{AloOrt96} could be successfully applied to systems of second order equations. 

Roughly speaking, we will show in Section \ref{sec-2} that a first result, Theorem \ref{teo-glob-1}, can be easily achieved
when appropriate global conditions are imposed on the coupling terms $h_1,\ldots,h_d$. 
In such a situation, the resonance assumption \eqref{res-intro} can be even weakened into
$$
n_j \in \mathbb{N} \quad \mbox{ for some } j \in \{1,\ldots,d\}.
$$
Moreover, with the same approach, it is possible to consider a genuinely vectorial problem like
$$
x'' + Ax + h(x) = p(t), \qquad x \in \mathbb{R}^d,
$$
where $h: \mathbb{R}^d \to \mathbb{R}^d$, $p: \mathbb{R}\to \mathbb{R}^d$ and $A$ is a symmetric, positive definite, $d \times d$ matrix which has $n^{2}$ as eigenvalue, with $n \in \mathbb{N}$
(see Theorem \ref{teo-glob-2}).

The possibility of obtaining results with asymptotic assumptions on the functions $h_1,\ldots,h_d$ is then studied in Section \ref{sec-3}.
As expected, this is a much more delicate issue, since the coupling between the equations plays an essential role. It seems then necessary 
to focus the attention on quite specific choices for the coupling terms $h_1,\ldots,h_d$; we will investigate in details the case of the cyclic coupling
\begin{equation}\label{cyc-intro}
h_j(x_{1},\dots,x_{d}) = h_j(x_{j+1})
\end{equation}
and of the radial coupling
\begin{equation}\label{rad-intro}
h_j(x_{1},\dots,x_{d}) = h_j\left(\sqrt{x_1^2 + \cdots + x_d^2}\right).
\end{equation}
We refer to Theorem \ref{teo-asy-1} and Theorem \ref{teo-asy-2} for the precise statements; notice that some care must be taken to describe the sets of solutions which are unbounded.

Let us finally recall that, for scalar second order equations at resonance, the existence of unbounded solutions
is strictly related to the existence of periodic solutions, and these problems are often considered together. In this regard, we mention that, despite some recent results obtained about periodic solutions for weakly coupled systems of ODEs (see, for instance, \cite{BosOrt14,FonSfe12,FonSfe16,FonUre17,LiuQia20} and the references therein), the existence of periodic solutions for systems like
\eqref{sys-intro} under the resonance assumption \eqref{res-intro} seems to be a quite open issue (we are just aware of \cite{AmsDen09}).

\medbreak
\noindent
\textbf{Notation.} Throughout the paper, the symbol $| \cdot |$ will be freely used to denote the absolute value of a real number, the modulus of a complex number or the Euclidean norm of a $k$-dimensional vector for $k \leq d$ (the specific meaning will be clear from the context).
We also denote $\mathbb{N}=\{1,2,3,\dots\}$.

\section{A global result}\label{sec-2}

In this section we prove the existence of unbounded solutions for system \eqref{sys-intro} under global assumptions on the coupling terms
$h_1,\ldots,h_d$. For briefness, from now on we write \eqref{sys-intro} in compact form as
\begin{equation}\label{sys-global}
x''_j + n_j^2 \,x_j + h_j(x) = p_j(t), \qquad j= 1,\ldots,d,
\end{equation}
where $ x =( x_{1},\dots,x_{d}) $, and we assume $n_j > 0$, $h_j: \mathbb{R}^d \to \mathbb{R}$ locally Lipschitz continuous and bounded and $p_j: \mathbb{R} \to \mathbb{R}$ continuous and $2\pi$-periodic (for every $j$).
Our result reads as follows.

\begin{Theorem} \label{teo-glob-1}
	In the previous setting, assume that 
	\begin{equation} \label{eq-ipo-glob-1}
	n_j \in \mathbb{N} \qquad \mbox{ and } \qquad 2 (\sup h_j-\inf h_j) < \left\vert \int_0^{2\pi} p_j(t) e^{\textnormal{i} n_j t}\,dt\right\vert\qquad \text{for some } j \in \{1,\ldots,d\}.
	\end{equation}
	Then, for every solution $x$ of \eqref{sys-global} it holds that
	\begin{equation} \label{eq-tesi-glob-1}
	\lim_{|t|\to +\infty} (x_j(t)^2+x'_j(t)^2)=+\infty.
	\end{equation}
\end{Theorem}

Let us observe that this result is, basically, of scalar nature: indeed, assumption \eqref{eq-ipo-glob-1} involves only the $j$-th equation of system \eqref{sys-global} and, accordingly, the conclusion is for the $j$-th component of the vector solution $(x_1,\ldots,x_d)$. Of course, in 
the case when system \eqref{sys-global} is fully at resonance (meaning that \eqref{res-intro} is satisfied), and the global assumption
$2 (\sup h_j-\inf h_j) < \left\vert \int_0^{2\pi} p_j(t) e^{\textnormal{i} n_j t}\,dt\right\vert$ holds for every $j = 1,\ldots,d$, then it follows that all the components of any solution of \eqref{sys-global} are unbounded, both in the past and in the future.

\begin{remark} \label{oss-coeffFourier}
	We recall the following fact (cf. \cite[Lemma 3.2]{AloOrt96}) which will be used several times throughout the paper: for every integrable function $q$ and every $n\in \mathbb{N}$, it holds that
	\begin{equation} \label{eq-coeffFourier}
	\left \vert \int_0^{2\pi} q(t) e^{\textnormal{i} n t}\,dt\right\vert=\max_{\varphi\in [0,2\pi]} \int_0^{2\pi} q(t)\sin (nt+\varphi)\,dt.
	\end{equation}
\end{remark}

We are now in a position to give the proof of Theorem \ref{teo-glob-1}.

\begin{proof}
Let $x$ be a solution of \eqref{sys-global}. We divide the proof in two steps.

\medskip
\noindent
\textit{Step 1.} Let $k\in \mathbb{Z}$; we show that
\begin{equation} \label{eq-interiglob-1}
\lim_{|k|\to +\infty} \left(x_j(2k\pi)^2+x'_j(2k\pi)^2\right)=+\infty.
\end{equation}
To this end, according to Remark \ref{oss-coeffFourier}, let $\varphi_j\in [0,2\pi]$ be such that
	\begin{equation} \label{eq-coeffFourier2}
\left \vert \int_0^{2\pi} p_j(t) e^{\textnormal{i} n_j t}\,dt\right\vert=\int_0^{2\pi} p_j(t)\sin (n_j t+\varphi_j)\,dt
\end{equation}
and let
\begin{equation} \label{eq-defgamma1}
\Gamma=\int_0^{2\pi} p_j(t)\sin (n_jt+\varphi_j)\,dt-2(\sup h_j-\inf h_j).
\end{equation}
In view of assumption \eqref{eq-ipo-glob-1}, it results $\Gamma >0$. 

Let us now define the Lyapunov-like function
\begin{equation} \label{eq-defv1}
V(\zeta,\eta)=\eta \sin \varphi_j-n_j\zeta \cos \varphi_j,\quad \forall \ (\zeta,\eta)\in \RR.
\end{equation}
A simple computation shows that for every integer $k>0$ we have
\begin{align*}
&V(x_j(2k\pi),x'_j(2k\pi))-V(x_j(0),x'_j(0)) =
\left[ x_{j}'(t)\sin(n_{j}t+\phi_{j})\right]_{0}^{2k\pi} - n_{j}\left[x_{j}(t)\cos(n_{j}t+\phi_{j})\vphantom{x_{j}'}\right]_{0}^{2k\pi}
\\
& \quad = \int_0^{2k\pi} \left[x''_j(t)+n_{j}^{2}x_{j}(t)\right]\sin (n_j t+\varphi_j)\,dt\\
& \quad =\int_0^{2k\pi} p_j(t) \sin (n_j t+\varphi_j)\,dt-\int_0^{2k\pi} h_j(x(t)) \sin (n_j t+\varphi_j)\,dt\\
& \quad \geq k\int_0^{2\pi} p_j(t) \sin (n_j t+\varphi_j)\,dt - \sup h_j\, \int_0^{2k\pi} \sin^+ (n_j t+\varphi_j)\,dt+
\inf h_j \, \int_0^{2k\pi} \sin^- (n_j t+\varphi_j)\,dt \\
&\quad =k\int_0^{2\pi} p_j(t) \sin (n_j t+\varphi_j)\,dt -2k (\sup h_j-\inf h_j),
\end{align*}
where $\sin^{+}s=\max\{\sin s, 0\} $ and $ \sin^{-}s=\max\{-\sin s, 0\}$.
Recalling \eqref{eq-defgamma1}, we then deduce that
\[
V(x_j(2k\pi),x'_j(2k\pi))-V(x_j(0),x'_j(0))\geq k\Gamma,\quad \forall \ k\in \mathbb{N}.
\]
This implies that
\[
\lim_{k\to +\infty} V(x_j(2k\pi),x'_j(2k\pi))=+\infty,
\]
thus proving \eqref{eq-interiglob-1} when $k\to +\infty$.

When $k$ is a negative integer, by arguing as above it is possible to prove that
\[
V(x_j(0),x'_j(0))-V(x_j(2k\pi),x'_j(2k\pi))\geq -k\Gamma,
\]
implying that
\[
\lim_{k\to -\infty} V(x_j(2k\pi),x'_j(2k\pi))=-\infty.
\]
and finally \eqref{eq-interiglob-1} when $k\to -\infty$.

\medskip
\noindent
\textit{Step 2.} We prove that
\begin{equation} \label{eq-tuttoinf1}
\lim_{|t|\to +\infty} \left(x_j(t)^2+x'_j(t)^2\right)=+\infty.
\end{equation}
To this end, let 
\[
M_j=\sup |h_j|+\sup |p_j|
\]
and define the energy function
\[
E_j(t)=\dfrac{1}{2} x'_j(t)^2+\dfrac{1}{2} n_j^2 x_j(t)^2+\dfrac{1}{2} M_j^2,\quad \forall  \ t\in \R.
\]
A simple computation shows that 
$$
E'_j(t) = [p_j(t)-h_j(x_1(t),\ldots,x_d(t)]x'_j(t),\quad \forall \ t\in \R,
$$
and thus, using the elementary inequality $|ab| \leq \frac12 (a^2 + b^2)$ for $a,b \in \mathbb{R}$,
$$
|E'_j(t)|\leq E_j(t),\quad \forall \ t\in \R.
$$
Gronwall's Lemma then yields 
\begin{equation} \label{eq-gronwall1}
E_j(t)\geq e^{2k\pi-t}\, E_j(2k\pi)\geq e^{-2\pi} \, E_j(2k\pi),\quad \forall \ t\in [2k\pi, 2(k+1)\pi),
\end{equation}
where $k$ is the integer part of $t/2\pi$. Taking into account \eqref{eq-interiglob-1} and the definition of $E_j$, from \eqref{eq-gronwall1} we deduce that
\[
\lim_{t\to +\infty} E_j(t)=+\infty,
\]
which implies \eqref{eq-tuttoinf1} when $t\to +\infty$.
The proof for $t\to -\infty$ is analogous.
\end{proof}

In the remaining part of this section, we show how to deal, using the same scheme of proof, with the more general system
\begin{equation}\label{sys-matrice}
x'' + Ax + h(x) = p(t), \qquad x \in \mathbb{R}^d,
\end{equation}
where $A$ is a $d \times d$ matrix, $h: \mathbb{R}^d \to \mathbb{R}^d$ is locally Lipschitz continuous and bounded, 
$p: \mathbb{R} \to \mathbb{R}^d$ is continuous and $2\pi$-periodic (of course, system \eqref{sys-global} enters this setting, 
with $A$ diagonal). Before stating the result, we recall that the linear homogeneous system
\begin{equation} \label{eq-omogeneo}
x''+Ax=0,
\end{equation}
has a nontrivial $2\pi$- periodic solution if and only if $A$ has an eigenvalue of the form $n^2$, for some $n\in \mathbb{N}$.
With this in mind, the following result holds true (in the statement, $\langle \cdot, \cdot \rangle$ stands for the Euclidean scalar product in
$\mathbb{R}^d$ and $| \cdot |$ for the associated norm). 

\begin{Theorem} \label{teo-glob-2}
In the previous setting, suppose that the matrix $A$ is symmetric and positive definite and assume that there there exists $n\in \mathbb{N}$ such that $n^2$ is an eigenvalue of $A$. 
Finally, suppose that there exists a nontrivial $2\pi$-periodic solution $v$ of \eqref{eq-omogeneo} satisfying
	\begin{equation} \label{eq-ipo-glob-2}
	\sup |h| \int_0^{2\pi} |v(t)|\,dt < \left\vert \int_0^{2\pi} \langle p(t),v(t)\rangle\,dt\right\vert.
	\end{equation}
	Then, for every solution $x$ of \eqref{sys-matrice} it holds that
	\begin{equation} \label{eq-tesi-glob-2}
	\lim_{|t|\to +\infty} (|x(t)|^2+|x'(t)|^2)=+\infty.
	\end{equation}
\end{Theorem}

\begin{proof} Let $x$ be a solution of \eqref{sys-matrice}. We first observe that it is sufficient to prove that
	\begin{equation} \label{eq-interi-glob-2}
	\lim_{|k|\to +\infty} (|x(2k\pi)|^2+|x'(2k\pi)|^2)=+\infty,
	\end{equation}
	where $k$ is an integer. Indeed, defining the energy function
	\[
	E(t)=\dfrac{1}{2}|x'(t)|^2+\dfrac{1}{2}\langle Ax(t),x(t) \rangle+\dfrac{1}{2}M^2,\quad \forall \ t\in \R,
	\]
	where $M=\sup (|p|+|h|)$, by arguing as in the second step of the proof of Theorem \eqref{teo-glob-1} it can be shown that \eqref{eq-interi-glob-2} implies \eqref{eq-tesi-glob-2} (here, we use the fact that $A$ is positive definite).
	
	We now show the validity of \eqref{eq-interi-glob-2} when $k\to +\infty$; the case of $k\to -\infty$ is similar. Assume then $k > 0$ and let
	\begin{equation} \label{eq-defgamma2}
	\Gamma'=\left\vert \int_0^{2\pi} \langle p(t),v(t)\rangle\,dt\right\vert-\sup |h| \int_0^{2\pi} |v(t)|\,dt;
	\end{equation}
	observe that $\Gamma'>0$ by \eqref{eq-ipo-glob-2}. Let us define
	\begin{eqnarray} \label{eq-defv2}
	V(\zeta,\eta)=\langle\eta,v(0)\rangle-\langle\zeta,v'(0)\rangle,\quad \forall \ (\zeta,\eta)\in \R^d\times \R^d.
	\end{eqnarray}
	As in the previous proof, a standard computation shows that
	\begin{align*}
	 V(x(2k\pi),x'(2k\pi))-V(x(0),x'(0))&=\int_0^{2k\pi} \langle x''(t)+Ax(t),v(t)\rangle,dt\\
	&=\int_0^{2k\pi} \langle p(t),v(t)\rangle\,dt-\int_0^{2k\pi} \langle h(x(t)),v(t)\rangle\,dt \\
	&\geq k\int_0^{2\pi} \langle p(t),v(t)\rangle\,dt-\sup |h|\,  \int_0^{2k\pi} |v(t)|\,dt\\
	&= k\int_0^{2\pi} \langle p(t),v(t)\rangle\,dt -k \sup |h| \, \int_0^{2\pi} |v(t)|\,dt.
	\end{align*}
	Recalling \eqref{eq-defgamma2}, we then deduce that
	\[
	V(x(2k\pi),x'(2k\pi))-V(x(0),x'(0))\geq k\Gamma',\quad \forall k \in \mathbb{N}.
	\]
	This implies that
	\[
	\lim_{k\to +\infty} V(x(2k\pi),x'(2k\pi))=+\infty,
	\]
	thus proving \eqref{eq-interi-glob-2} when $k\to +\infty$.
\end{proof}

It is worth noticing that, with the same proof, Theorem~\ref{teo-glob-2} could be extended to the case of an infinite-dimensional system like
$$
x'' + Ax + h(x) = p(t), \qquad x \in \mathcal{H},
$$
where $\mathcal{H}$ is a real Hilbert space and $A: \mathcal{H} \to \mathcal{H}$ is a positive, bounded and self-adjoint linear operator.
For some recent advances on the topic of ODEs in Hilbert spaces, see \cite{BosFonGar->,FonKluSfePP,FonMawWil->} and the references therein.

We end this section with a comparison between Theorem \ref{teo-glob-1} and Theorem \ref{teo-glob-2} in the case when $A$ is a diagonal matrix.
\begin{remark}
	Theorem \ref{teo-glob-2} is of vectorial nature. In the case of a diagonal matrix $A$, when compared with Theorem \eqref{teo-glob-1} it requires a stronger condition in the coupling term $h$. Indeed, if $n_j^2$ is the eigenvalue of $A$ corresponding to an integer $n_j$, we can take
	\[
	v(t)=\sin (n_jt+\varphi_j) e_j,\quad \forall \ t\in \R,
	\] 
	where $\varphi_j$ is as in \eqref{eq-coeffFourier2} and $e_j$ is the $j$-th vector of the standard basis of $\R^d$: hence, assumption \eqref{eq-ipo-glob-2} reads as
	\[
	\int_0^{2\pi} p_j(t)\sin (n_j t+\varphi_j)\,dt=\left \vert \int_0^{2\pi} p_j(t) e^{\textnormal{i} n_j t}\,dt\right\vert > \sup |h| \int_0^{2\pi} |\sin (n_jt+\varphi_j)|\,dt=4\sup |h|,
	\]
	which is stronger than \eqref{eq-ipo-glob-1} since
	\[
	4\sup |h|\geq 4\sup |h_j|\geq 2 (\sup h_j-\inf h_j).
	\]
	On the other hand, a sharper result in the general case of system \eqref{sys-matrice} can be proved by diagonalizing the matrix $A$. Indeed, let $Q$ be an orthogonal matrix such that $Q A Q^T=D$ is diagonal and let $y=Qx$. The original system \eqref{sys-matrice} is then transformed in
	\begin{equation} \label{eq-nuovosis}
	y''+Dy+h^*(y)=p^*(t),
	\end{equation}
	where $h^*(y)=Q h(Q^Ty) $ and $p^*(t)=Q p(t)$.
As a consequence, an unboundedness result for $x$ can be obtained by applying Theorem \eqref{teo-glob-1} to \eqref{eq-nuovosis} when 
	\[
	2(\sup h^*_j-\inf h^*_j)<\left\vert \int_0^{2\pi} p^*_j(t)e^{\textnormal{i} n_j t}\,dt\right \vert.
	\]
	This is a sharper assumption compared to \eqref{eq-ipo-glob-2}, but it involves the matrix $Q$, whose knowledge is not required in \eqref{eq-ipo-glob-2}.
\end{remark}

\section{Two asymptotic results}\label{sec-3}

In this section, we deal again with system \eqref{sys-global}, by studying a couple of situations in which the assumptions on the nonlinear coupling term 
involve its asymptotic behavior rather than the span of its image.
Whenever $n_{j} \in \mathbb{N} $, we will use again (cf. \eqref{eq-defv1}) the Lyapunov function
$ V_{j,\varphi}(\zeta,\eta) = \zeta\sin\varphi - n_{j}\eta\cos\varphi $, with suitable choiches of $\varphi\in[0,2\pi]$, to estimate the growth of $(x_{j}(2k\pi),x'_{j}(2k\pi))$ with respect to $ k\in\Z $.
As in the proof of Theorem~\ref{teo-glob-1}, if $ n_{j}\in\mathbb{N} $, then we have that
\begin{equation}\label{eq-deltaV}
V_{j,\varphi}(x_{j}(2\pi),x'_{j}(2\pi)) - V_{j,\varphi}(x_{j}(0),x'_{j}(0)) =
\int_{0}^{2\pi}[ p_{j}(t) - h_{j}(x(t)) ]\sin(n_{j}t+\varphi)\,dt
\end{equation}
and that there exists $ \varphi_{j}^{0} \in [0,2\pi] $ such that
\[
\int_{0}^{2\pi} p_{j}(t) \sin\left( n_{j}t+\varphi_{j}^{0} \right)\,dt = \left| \int_{0}^{2\pi} p_{j}(t) e^{in_{j}t}\,dt \right|.
\]
The estimate the term $\int_{0}^{2\pi}h_{j}(x(t)) \sin(n_{j}t+\varphi)\,dt$ for large solutions
will depend on the particular form of the nonlinear term $h_{j}$: we consider cyclic and radial dependencies.
In any case, we need an estimate of the solution  on $[0,2\pi]$ with respect to initial conditions.
A straightforward application of the variation of constants formula shows that
\[
\begin{aligned}
x_{j}(t) &= x_{j}(0) \cos n_{j}t + \frac{x'_{j}(0)}{n_{j}}\sin n_{j}t + 
\frac{1}{n_{j}}\int_{0}^{t} \left[ p_{j}(s)-h_{j}(x(s))\right]\sin n_{j}(t-s)\,ds 
\end{aligned} \qquad j=1,\dots,d
\]
if $x$ is any solution of our system.
In particular, the boundedness of $ h_j $ and $ p_j $ implies the following estimate:
\begin{equation}\label{eq-stimaC2}
\begin{aligned}
x_{j}(t) &= A_{j} \sin(n_{j}t+\omega_{j}) + \sigma_{j}(t) \\
A_{j}    &=\sqrt{x_{j}^{2}(0)+\frac{x'_{j}(0)^2}{n_{j}^{2}}}, \quad
\omega_{j} \in [0,2\pi], \quad \|\sigma_{j}\|_{C^{2}([0,2\pi])}\le C,
\end{aligned}
\end{equation}
where the constant $C$ is independent of $A_{j}$ and $j$.
\subsection{Systems with cyclic coupling}
Here we assume that system \ref{sys-global} is cyclically coupled, that is
$$
h_{j}(x) = h_{j}(x_{j+1})
$$
for each $ j \in \{1,\ldots,d\}$ (as usual, we use the cyclic agreement $x_{d+1} = x_1$); cf. \eqref{cyc-intro}.
In particular, we have that $h_j: \mathbb{R} \to \mathbb{R}$ is locally Lipschitz continuous and bounded
for each $j\in\{1,\ldots,d\}$.
Henceforth, we will thus write system \eqref{sys-global} as
\begin{equation}\label{sys-cyc}
x''_j + n_j^2 \, x_j + h_j(x_{j+1}) = p_j(t), \qquad j\in\{1,\ldots,d\}.
\end{equation}
For this subsection, we also introduce the notations
\[
\overline{h}_j(\pm \infty) = \limsup_{s \to \pm \infty} h_j(s), \qquad \underline{h}_j(\pm \infty) = \liminf_{s \to \pm \infty} h_j(s),
\]
and
\[
\Delta h_j = \max\{ \overline{h}_j(+\infty),\overline{h}_j(-\infty) \} - \min \{ \underline{h}_j(+\infty),\underline{h}_j(-\infty)\}.
\]
Notice that, when $h_j$ has limits $h_j(\pm \infty)$ at infinity, then
\[
\Delta h_j = \left\vert h_j(+\infty) - h_j(-\infty) \right\vert.
\]
Our result reads as follows.

\begin{Theorem}\label{teo-asy-1}
In the previous setting, assume that
\[
n_j \in \mathbb{N} \qquad \mbox{ and } \qquad 2 \Delta h_j < \left\vert \int_0^{2\pi} p_j(t) e^{\textnormal{i} n_j t}\,dt\right\vert,\qquad \text{for every } j \in \{1,\ldots,d\}.
\]
Then, for every $j \in\{1,\ldots,d\}$ there exist open sets $\mathcal{C}^-_j, \mathcal{C}^+_j \subset \mathbb{R}^2$ such that $ \RR\setminus ( \mathcal{C}_{j}^{+}\cup\mathcal{C}_{j}^{-} ) $ is compact and, setting 
\[
\mathcal{C} := \prod_{j=1}^d \left( \mathcal{C}^-_j \cup \mathcal{C}^+_j \right),
\]
the following holds true:
\begin{itemize}
\item[(i)] for every solution $x$ of \eqref{sys-cyc},
\[
(x(0),x'(0)) \in \mathcal{C} \; \Longrightarrow \;\sup_{t \in \mathbb{R}} \left( x_j(t)^2 + x'_{j}(t)^2 \right) = +\infty, \quad \mbox{ for every } j \in\{1,\ldots,d\}.
\]
\end{itemize}
Moreover, setting
\[
\mathcal{C}^- =  \prod_{j=1}^d \mathcal{C}^-_j \quad \mbox{ and } \quad \mathcal{C}^+ =  \prod_{j=1}^d \mathcal{C}^+_j,
\]
the following hold:
\begin{itemize}
\item[(ii)] $\mathcal{C}^- \cap \mathcal{C}^+$ has infinite measure;
\item[(iii)] for every solution $x$ of \eqref{sys-cyc},
\[
(x(0),x'(0)) \in \mathcal{C}^- \; \Longrightarrow \;\lim_{t \to -\infty} \left( x_j(t)^2 + x'_{j}(t)^2 \right) = +\infty, \quad \mbox{ for every } j\in\{1,\ldots,d\},
\] 
\[
(x(0),x'(0)) \in \mathcal{C}^+ \; \Longrightarrow \;\lim_{t \to +\infty} \left( x_j(t)^2 + x'_{j}(t)^2 \right) = +\infty, \quad \mbox{ for every } j\in\{1,\ldots,d\}.
\] 
\end{itemize}
\end{Theorem}

Let us explain more informally the meaning of Theorem \ref{teo-asy-1}.
To this end, notice that
\[
\mathcal{C}^- \cap \mathcal{C}^+ \subset \mathcal{C}^{\pm} \subset \mathcal{C}.
\]
The set $\mathcal{C}$ is the largest one: its projection $ \mathcal{C}_{j}^{-}\cup\mathcal{C}_{j}^{+}$ on the $j$-th factor $\RR$ has compact complement.
The sets $\mathcal{C}^{\pm}$ are smaller (and indeed it will be clear from the proof that their components $ \mathcal{C}_{j}^{\pm} $ are not the complement of a compact set) but they have the property of having infinite measure: even more, their intersection $\mathcal{C}^- \cap \mathcal{C}^+$ has infinite measure, as well.
The unbounded properties of the solutions starting on this sets can be described as follows:
\begin{itemize}
\item[(i)] for vector solutions $x$ starting on $\mathcal{C}$, each component $x_j$ is unbounded in the phase-plane: it could be, however, that some component is unbounded in the past and some other in the future;
\item[(ii)] for vector solutions $x$ starting on $\mathcal{C}^-$ (resp., $\mathcal{C}^+$),
each component $x_j$ is unbounded in the past (resp., in the future);
\item[(iii)] for vector solutions $x$ starting on $\mathcal{C}^- \cap \mathcal{C}^+$,
each component $x_j$ is unbounded both in the past and in the future.
\end{itemize}

\begin{proof}
We fix $ \Gamma > 0 $ such that
\[
\int_{0}^{2\pi} p_{j}(t) \sin\left( n_{j}t+\varphi_{j}^{0}\right)\,dt = \left| \int_{0}^{2\pi} p_{j}(t) e^{\text{i}n_{j}t}\,dt \right|
> 2( \Delta h_{j} + \Gamma ), \qquad j \in \{1,\ldots,d\} .
\]
By continuity, there exist open intervals $ I_{j}$ containing $ \varphi_{j}^{0} $ such that
\begin{equation}\label{eq-Fourier}
\int_{0}^{2\pi} p_{j}(t) \sin( n_{j}t+\varphi_{j} )\,dt > 2( \Delta h_{j} + \Gamma ),
\qquad \forall \varphi_{j}\in I_{j} \text{ and } j\in\{1,\ldots,d\}.
\end{equation}
We claim that:
\begin{equation}\label{eq-Fatou}
\limsup_{A_{j+1}\to+\infty} \int_{0}^{2\pi}h_{j}(x_{j+1}(t))\sin(n_{j}t+\varphi)\,dt \le 2\Delta h_{j},
\quad\text{uniformly w.r.t. } \varphi\in[0,2\pi] \text{ and } A_{i}, i\ne j+1,
\end{equation}
where we recall that $ A_{j+1}\to+\infty $ is equivalent to $\left|(x_{j+1}(0),x'_{j+1}(0))\right|\to+\infty$,
as it can be seen in \eqref{eq-stimaC2}.
Indeed, by Fatou's lemma we have that
\[
\limsup_{A_{j+1}\to+\infty} \int_{0}^{2\pi}h_{j}(x_{j+1}(t))\sin(n_{j}t+\varphi)\,dt \le
\int_{0}^{2\pi}\limsup_{A_{j+1}\to+\infty} [ h_{j}(x_{j+1}(t))\sin(n_{j}t+\varphi) ]\,dt
\]
and estimate \eqref{eq-stimaC2} shows that for almost every $ t \in [0,\pi] $ we have that
either $ x_{j+i}(t)\to +\infty $ or $ x_{j+i}(t)\to -\infty $ as $ A_{j+1}\to +\infty$ and these limits
hold uniformly w.r.t. $ A_{i} $ for $ i\ne j+1$.
Therefore, we deduce that:
\[
\begin{aligned}
\limsup_{A_{j+1}\to+\infty} & [ h_{j}(x_{j+1}(t))\sin(n_{j}t+\varphi) ] \\ 
&= \sin^{+}(n_{j}t+\varphi)\limsup_{A_{j+1}\to+\infty} h_{j}(x_{j+1}(t))
- \sin^{-}(n_{j}t+\varphi)\liminf_{A_{j+1}\to+\infty} h_{j}(x_{j+1}(t)) \\
& \le \sin^{+}(n_{j}t+\varphi) \max\{\overline{h}_{j}(+\infty),\overline{h}_{j}(-\infty)\}
- \sin^{-}(n_{j}t+\varphi) \min\{\underline{h}_{j}(+\infty),\underline{h}_{j}(-\infty)\}
\end{aligned}
\]
for almost every $ t\in [0,2\pi] $.
Since $ \int_{0}^{2\pi}\sin^{\pm}(n_{j}t+\varphi)\,dt = 2 $, we obtain \eqref{eq-Fatou}.

Thanks to \eqref{eq-Fatou}, there is a constant $ A>0 $ such that
\[
x_{j+1}(0)^2+x'_{j+1}(0)^2\ge A^{2} \implies
\int_{0}^{2\pi}h_{j}(x_{j+1}(t))\sin(n_{j}t+\varphi)\,dt\le 2\Delta h_{j}+\Gamma, 
\qquad \forall \varphi\in[0,2\pi],
\]
for each $ j\in\{1,\ldots,d\} $.
Hence, we deduce that
\[
x_{j+1}(0)^2+x'_{j+1}(0)^2\ge A^{2} \implies V_{j,\varphi_{j}}(x_{j}(2\pi),x'_{j}(2\pi)) - V_{j,\varphi_{j}}(x_{j}(0),x'_{j}(0)) \ge \Gamma,
\qquad \forall \varphi_{j} \in I_{j},
\]
by using \eqref{eq-deltaV} and \eqref{eq-Fourier}.

Let us define 
$$
 \overline{V}_{j}(A) = \max\{ V_{j,\varphi}(\zeta,\eta):|(\zeta,\eta)|\le A \text{ and } \varphi\in[0,2\pi] \} .
$$
If the initial conditions satisfy $ V_{j,\varphi_{j}}(x_{j}(0),x'_{j}(0)) > \overline{V}_{j}(A) $ for
 some $\varphi_{j}\in I_{j} $ and all $ j\in\{1,\ldots,d\} $,
then we have that $ x_{j}(0)^2+x'_{j}(0)^2\ge A^{2} $ for all $ j\in\{1,\ldots,d\} $ and, hence,
\[
V_{j,\varphi_{j}}(x_{j}(2\pi),x'_{j}(2\pi))\ge V_{j,\varphi_{j}}(x_{j}(0),x'_{j}(0)) +\Gamma
> V_{j,\varphi_{j}}(x_{j}(0),x'_{j}(0)) \ge A, \qquad \forall j\in\{1,\ldots,d\}.
\]
Therefore, we obtain recursively that
\[
V_{j,\varphi_{j}}(x_{j}(2k\pi),x'_{j}(2k\pi)) - V_{j,\varphi_{j}}(x_{j}(0),x'_{j}(0)) \ge k\Gamma,
\qquad \forall k\in\mathbb{N} \text{ and } j \in \{1,\ldots,d\},
\]
which implies that $ x_{j}(2k\pi)^2 + x'_{j}(2k\pi)^2 \to +\infty $ as $ k \to +\infty $.

In a similar way, if the initial conditions satisfy $ V_{j,\varphi_{j}}(x_{j}(0),x'_{j}(0)) < -\overline{V}_{j}(A) $ for  some $\varphi_{j}\in I_{j} $ and all $ j\in\{1,\ldots,d\} $, then we can show that
$ x_{j}(2k\pi)^2 + x'_{j}(2k\pi)^2 \to +\infty $ as $ k \to -\infty $.

We observe now that, for a fixed $j\in\{1,\ldots,d\}$ and a fixed $\varphi_{j}\in I_{j} $, the inequalities
 $ V_{j,\varphi_{j}}(\zeta,\eta) < -\overline{V}_{j}(A) $ and $ V_{j,\varphi_{j}}(\zeta,\eta) > \overline{V}_{j}(A) $ define two opposite half-planes in $\RR$ that originate from the two parallel
straight lines $V_{j,\varphi_{j}}(\zeta,\eta) = \pm \overline{V}_{j}(A) $.
For each $j\in\{1,\ldots,d\}$ we fix two angles $ \varphi_{j}^{1}, \varphi_{j}^{2} \in I_{j} $ such that
$ 0 < | \varphi_{j}^{1} - \varphi_{j}^{2} | < \pi $ and define the open regions
\[
\begin{aligned}
\mathcal{C}_{j}^{+} & = \{ (\zeta,\eta)\in\RR : V_{j,\varphi_{j}^{1}}(\zeta,\eta) > \overline{V}_{j}(A) \}
\cup \{ (\zeta,\eta)\in\RR : V_{j,\varphi_{j}^{2}}(\zeta,\eta) > \overline{V}_{j}(A) \} \\
\mathcal{C}_{j}^{-} & = \{ (\zeta,\eta)\in\RR : V_{j,\varphi_{j}^{1}}(\zeta,\eta) < -\overline{V}_{j}(A) \}
\cup \{ (\zeta,\eta)\in\RR : V_{j,\varphi_{j}^{2}}(\zeta,\eta) < -\overline{V}_{j}(A) \}, \\
\end{aligned}
\]
which are actually two reflex angles in the plane and opposite to each other with respect to the origin.
In particular $ \mathcal{C}_{j}^{+}\cap\mathcal{C}_{j}^{-} $ is made up by two opposite and disjoint angles and, thus, has infinite measure, while $ \RR\setminus(\mathcal{C}_{j}^{+}\cup\mathcal{C}_{j}^{-}) $ is a compact parallelogram.

We conclude by using the argument at the beginning of the proof of Theorem~\ref{teo-glob-1}.
\end{proof}
\subsection{Systems with radial coupling}
We assume here that the nonlinear term depends only on the Euclidean norm of the vector solution, that is
$$
 h_{j}(x) = h_{j}(|x|), \qquad j\in\{1,\ldots,d\};
$$
cf. \eqref{rad-intro}. Therefore, system \eqref{sys-global} becomes 
\begin{equation}\label{sys-rad}
x''_{j} + n_{j}^2 \, x_{j} + h_{j}\left(|x|\right) = p_{j}(t), \qquad j\in\{1,\ldots,d\},
\end{equation}
where $h_{j}: [0,+\infty) \to \mathbb{R}$ is locally Lipschitz continuous and bounded, for $j\in\{1,\ldots,d\}$.
In this subsection we also set
\[
\overline{h}_{j}(+\infty) = \limsup_{s \to +\infty} h_{j}(s), \qquad
\underline{h}_{j}(+\infty) = \liminf_{s \to +\infty} h_{j}(s) \quad \text{and}\quad
\Delta h_{j} = \overline{h}_{j}(+\infty) - \underline{h}_{j}(+\infty), \qquad \forall j\in\{1,\ldots,d\}.
\]
Our result reads as follows.
\begin{Theorem}\label{teo-asy-2}
In the previous setting, assume that
\[
n_{j} \in \mathbb{N} \qquad \text{and} \qquad 2 \Delta h_{j} < \left\vert \int_0^{2\pi} p_{j}(t) e^{\textnormal{i} n_{j} t}\,dt\right\vert,\qquad \text{for some } j \in \{1,\ldots,d\}.
\]
Then, there exist open sets $\mathcal{C}^-_{j}, \mathcal{C}^+_{j} \subset \mathbb{R}^2$ such that 
$\mathbb{R}^2 \setminus (\mathcal{C}^-_{j} \cup \mathcal{C}^+_{j})$ is compact, 
$\mathcal{C}^- _j\cap \mathcal{C}^+_j$ has infinite measure and, moreover,
the following holds true for every solution $x$ of \eqref{sys-rad}:
\[
\begin{aligned}
(x_{j}(0),x'_{j}(0)) \in \mathcal{C}_{j}^-
& \Longrightarrow \lim_{t \to -\infty} \left( x_{j}(t)^2 + x'_{j}(t)^2 \right) = +\infty; \\
(x_{j}(0),x'_{j}(0)) \in \mathcal{C}_{j}^+
& \Longrightarrow \lim_{t \to +\infty} \left( x_{j}(t)^2 + x'_{j}(t)^2 \right) = +\infty.
\end{aligned}
\]
\end{Theorem}

Let us notice that, contrarily to Theorem \ref{teo-asy-1}, both the assumptions and the conclusions of Theorem \ref{teo-asy-2} refer to a component $x_j$ of the vector solution $x$ (from this point of view, the statement is more similar to the one of Theorem \ref{teo-glob-1}).

\begin{proof}
As in the proof of Theorem~\ref{teo-asy-1}, we fix $ \Gamma > 0$ and an open interval $ I_{j} \subset [0,2\pi]$  such that 
\begin{equation}\label{eq-FourierRad}
\int_{0}^{2\pi} p_{j}(t) \sin( n_{j}t+\varphi_{j} )\,dt > 2( \Delta h_{j} + \Gamma ),
\qquad \forall \varphi_{j}\in I_{j}.
\end{equation}
We claim that
\begin{equation}\label{eq-FatouRad}
\limsup_{A_{j}\to+\infty} \int_{0}^{2\pi}h_{j}(|x(t)|)\sin(n_{j}t+\varphi)\,dt \le 2\Delta h_{j},
\quad \text{uniformly w.r.t. } \varphi\in[0,2\pi] \text{ and } A_{i}, \text{ for } i\ne j.
\end{equation}
Again, by Fatou's lemma we have that
\[
\limsup_{A_{j}\to+\infty} \int_{0}^{2\pi}h_{j}(|x(t)|)\sin(n_{j}t+\varphi)\,dt \le
\int_{0}^{2\pi}\limsup_{A_{j}\to+\infty} [ h_{j}(|x(t)|)\sin(n_{j}t+\varphi) ]\,dt
\]
and estimate \eqref{eq-stimaC2} shows that for almost every $ t \in [0,\pi] $ we have that
$ |x(t)| \ge |x_{j}(t)|\to +\infty $ as $ A_{j}\to +\infty$ (still uniformly w.r.t. $A_{i}$, $i\ne j$).
Therefore, we deduce that
\[
\limsup_{A_{j}\to+\infty} [ h_{j}(|x(t)|)\sin(n_{j}t+\varphi) ] 
= \sin^{+}(n_{j}t+\varphi) \overline{h}_{j}(+\infty)
- \sin^{-}(n_{j}t+\varphi) \underline{h}_{j}(+\infty)
\]
for almost every $ t\in [0,2\pi] $ and  we obtain \eqref{eq-FatouRad}.

Thanks to \eqref{eq-FatouRad}, there is a constant $ A>0 $ such that
\[
x_{j}(0)^2+x'_{j}(0)^2\ge A^{2} \implies
\int_{0}^{2\pi}h_{j}(|x(t)|)\sin(n_{j}t+\varphi)\,dt\le 2\Delta h_{j}+\Gamma, 
\qquad \forall \varphi\in[0,2\pi].
\]
Hence, we deduce that
\[
x_{j}(0)^2+x'_{j}(0)^2\ge A^{2} \implies V_{j,\varphi_{j}}(x_{j}(2\pi),x'_{j}(2\pi)) - V_{j,\varphi_{j}}(x_{j}(0),x'_{j}(0)) \ge \Gamma,
\qquad \forall \varphi_{j} \in I_{j},
\]
by using \eqref{eq-deltaV} and \eqref{eq-FourierRad}.

Let us define $ \overline{V}_{j}(A) = \max\{ V_{j,\varphi}(\zeta,\eta):|(\zeta,\eta)|\le A \text{ and } \varphi\in[0,2\pi] \} $.
If the initial conditions satisfy $ V_{j,\varphi_{j}}(x_{j}(0),x'_{j}(0)) > \overline{V}_{j}(A) $ for
 some $\varphi_{j}\in I_{j} $,
then we have that $ x_{j}(0)^2+x'_{j}(0)^2\ge A^{2} $ and, hence,
\[
V_{j,\varphi_{j}}(x_{j}(2\pi),x'_{j}(2\pi))\ge V_{j,\varphi_{j}}(x_{j}(0),x'_{j}(0)) +\Gamma
> V_{j,\varphi_{j}}(x_{j}(0),x'_{j}(0)) \ge A.
\]
Therefore, we obtain recursively that
\[
V_{j,\varphi_{j}}(x_{j}(2k\pi),x'_{j}(2k\pi)) - V_{j,\varphi_{j}}(x_{j}(0),x'_{j}(0)) \ge k\Gamma,
\qquad \forall k\in\mathbb{N},
\]
which implies that $ x_{j}(2k\pi)^2 + x'_{j}(2k\pi)^2 \to +\infty $ as $ k \to +\infty $.

In a similar way, if the initial conditions satisfy $ V_{j,\varphi_{j}}(x_{j}(0),x'_{j}(0)) < -\overline{V}_{j}(A) $ for  some $\varphi_{j}\in I_{j} $, then we can show that
$ x_{j}(2k\pi)^2 + x'_{j}(2k\pi)^2 \to +\infty $ as $ k \to -\infty $.

Now, we define the open regions
\[
\begin{aligned}
\mathcal{C}_{j}^{+} & = \{ (\zeta,\eta)\in\RR : V_{j,\varphi_{j}^{1}}(\zeta,\eta) > \overline{V}_{j}(A) \}
\cup \{ (\zeta,\eta)\in\RR : V_{j,\varphi_{j}^{2}}(\zeta,\eta) > \overline{V}_{j}(A) \}, \\
\mathcal{C}_{j}^{-} & = \{ (\zeta,\eta)\in\RR : V_{j,\varphi_{j}^{1}}(\zeta,\eta) < -\overline{V}_{j}(A) \}
\cup \{ (\zeta,\eta)\in\RR : V_{j,\varphi_{j}^{2}}(\zeta,\eta) < -\overline{V}_{j}(A) \},
\end{aligned}
\]
and we conclude as in the proof of Theorem~\ref{teo-asy-1}.
\end{proof}

\vspace{1.5cm}

\noindent Authors' addresses:

\bigbreak

\begin{tabular}{l}
Alberto Boscaggin\\
Dipartimento di Matematica\\
Universit\`a di Torino\\
Via Carlo Alberto 10, I-10123 Torino, Italy\\
e-mail: alberto.boscaggin@unito.it\\
\\
Walter Dambrosio\\
Dipartimento di Matematica\\
Universit\`a di Torino\\
Via Carlo Alberto 10, I-10123 Torino, Italy\\
e-mail: walter.dambrosio@unito.it\\
\\
Duccio Papini\\
Dipartimento di Matematica, Informatica e Fisica\\
Universit\`a di Udine\\
Via delle Scienze 206, I-33100 Udine, Italy\\
e-mail: duccio.papini@uniud.it
\end{tabular}

\end{document}